\documentclass[review,onefignum,onetabnum]{siamart171218}

\usepackage{lipsum}
\usepackage{amsfonts}
\usepackage{graphicx}
\usepackage{epstopdf}
\usepackage{algorithmic}
\ifpdf
  \DeclareGraphicsExtensions{.eps,.pdf,.png,.jpg}
\else
  \DeclareGraphicsExtensions{.pdf}
\fi

\newsiamremark{remark}{Remark}
\newsiamremark{assumption}{Assumption}
\newsiamremark{hypothesis}{Hypothesis}
\crefname{hypothesis}{Hypothesis}{Hypotheses}
\newsiamthm{claim}{Claim}

\headers{Multifidelity Convergence Acceleration}{V. Keshavarzzadeh, R. M. Kirby, and A. Narayan}

\title{Convergence Acceleration for Time-Dependent Parametric Multifidelity Models\thanks{Submitted for publication.
\funding{This research was partially sponsored by ARL under Cooperative Agreement Number W911NF-12-2-0023. The views and conclusions contained in this document are those of the authors and should not be interpreted as representing the official policies, either expressed or implied, of ARL or the U.S. Government. The U.S. Government is authorized to reproduce and distribute reprints for Government purposes notwithstanding any copyright notation herein. The second author is partially supported by DARPA TRADES HR0011-17-2-0016. The first and third authors are partially supported by AFOSR FA9550-15-1-0467. The third author is partially supported by DARPA EQUiPS N660011524053.}}}

\author{Vahid Keshavarzzadeh\thanks{Scientific Computing and Imaging Institute, University of Utah, Salt Lake City, UT
  (\email{vkeshava@sci.utah.edu}, \url{https://sites.google.com/view/vahid-keshavarzzadeh}).}
\and Robert M. Kirby\thanks{School of Computing, University of Utah, Salt Lake City, UT
  (\email{kirby@sci.utah.edu}, \url{https://www.cs.utah.edu/\string~kirby/}).}
\and Akil Narayan\thanks{Department of Mathematics, University of Utah, Salt Lake City, UT
  (\email{akil@sci.utah.edu}, \url{https://www.sci.utah.edu/\string~akil/}).}
}

\usepackage{amsopn}

\usepackage{amsmath,amssymb,mathtools,bbm,mathabx}

\newcommand{\bs}[1]{\boldsymbol{#1}}

\usepackage{array,booktabs} 

\usepackage{mathdefs}

\nolinenumbers
\begin{document}

\maketitle

\begin{abstract}
We present a numerical method for convergence acceleration for multifidelity models of parameterized ordinary differential equations. The hierarchy of models is defined as trajectories computed using different timesteps in a time integration scheme. 
  Our first contribution is in novel analysis of the multifidelity procedure, providing a convergence estimate. Our second contribution is development of a three-step algorithm that uses multifidelity surrogates to accelerate convergence: step one uses a multifidelity procedure at three levels to obtain accurate predictions using inexpensive (large timestep) models. Step two uses high-order splines to construct continuous trajectories over time. Finally, step three combines spline predictions at three levels to infer an order of convergence and compute a sequence transformation prediction (in particular we use Richardson extrapolation) that achieves superior error.  We demonstrate our procedure on linear and nonlinear systems of parameterized ordinary differential equations.
\end{abstract}



\section{Introduction}\label{sec:intro}

We investigate time-dependent models arising from parameterized ordinary differential equations (ODE). Such models arise in, for example, applied uncertainty quantification contexts. The following parameterized ODE defines the unknown $u$:
\begin{align}\label{eq:ode}
  \dfdx{u}{t}(t,k) &= f(t,u,k), & u(0, k) &= u_0(k),
\end{align}
where $u \in \R^M$ is a vector-valued state variable, $u_0 \in \R^M$ is a given initial condition, $k \in \R^d$ is a Euclidean parameter, and we take the time variable to range over $[0,T]$. The right-hand side function $f: [0,T] \times \R^M \times \R^d$ is also given. We assume the above system is well-posed for all $k$; in particular, we will codify some assumptions in Section \ref{sec2} so that the solution trajectory $u(\cdot,k)$ is smooth and so that standard discrete-time integration methods (e.g., multi-step and multi-stage methods) provide convergent approximations for fixed $T$.

The technique we adopt was proposed in \cite{Narayan14,zhu_computational_2014} and begins with the following approximation:
\begin{align}\label{eq:ut-approx}
  u(t,k) \approx \sum_{q=1}^n u(t, k_q) v_q(k),
\end{align}
where $n$ is small (in practice we use $n = \mathcal{O}(10)$), the $u(t,k_q)$ are discrete-time solution ``snapshots" at fixed parameter values computed using a refined timestep, and $v_q$ are computed from a coarse timestep approximation. The approximation above requires $n$ stored solutions computed using a refined timestep, and a single coarse timestep solution for each value of $k$. The parameter values $k_q$ and the parametric functions $v_q$ are computed via an analysis of coarse time discretizations. Thus, the entire procedure uses time discretizations at different discrete-time refinements (``fidelities"). Once the $n$ solutions $u(t,k_q)$ are stored, then evaluation of \eqref{eq:ut-approx} at a particular $k$ requires only one solution of the coarse timestep model.

Assuming solution trajectories are smooth, we supplement the multifidelity procedure above with two additional steps: Once the approximation above is constructed, we extend the discrete time solutions to continuous time via spline interpolation, and with spline representations on hand for each fidelity level we perform sequence transformations (e.g., Richardson Extrapolation) to accelerate convergence.

Our novel contributions are the derivation of mathematical error estimates that prove convergence of the approximation \eqref{eq:ut-approx}, and in development of computational algorithms that utilize spline representations and sequence transformations to accelerate convergence. An overview of the algorithm and our theoretical statements is provided below.

\subsection{Multifidelity algorithm overview}\label{sec:intro-mf}

We compute the coefficient functions $v_q$ in \eqref{eq:ut-approx} via a multifidelity procedure. Our models of different fidelities are outputs from discrete-time integration methods using different time steps. Let $r > 1$ be an integer, and let $h > 0$ be a stepsize at the coarsest level. We construct three discrete models, defined as
\begin{itemize}
  \item $u_1(\cdot,k)$ : discrete-time solution computed using a time step $h$, a ``low-fidelity" model. $u_1$ is inexpensive to compute for each $k$.
  \item $u_2(\cdot,k)$ : discrete-time solution computed using a time step $h/r$, a ``medium-fidelity" model. $u_2$ is moderately expensive to compute for each $k$.
  \item $u_3(\cdot,k)$ : discrete-time solution computed using a time step $h/r^2$, a ``high-fidelity" model. $u_3$ is expensive to compute for each $k$.
\end{itemize}
Our procedure performs an analysis of several trajectories of the inexpensive model $u_1$ to (i) identify the parameter values $\{k_q\}_{q=1}^n$ and (ii) compute the coefficient functions $v_q$ for use in \eqref{eq:ut-approx}. Precisely, the $v_q$ are defined as
\begin{align*}
  \left\{ v_1(k), \ldots, v_q(k) \right\} &= \argmin_{\bs{w} \in \R^q} \left\| u_1(\cdot,k) - \sum_{j=1}^q w_j u_1(\cdot, k_j) \right\| 
\end{align*}
where $\|\cdot\|$ is an appropriate $\ell^2$-type norm so that the $v_j$ can be easily computed as the solution to a linear least-squares problem given the data $u_1(t,k)$.\footnote{The values $\{v_1(k), \ldots, v_q(k)\}$ depend on the value of $q$, and so we are committing a small notational crime by not explicitly indexing the $v_j(k)$ by $q$.} The values $k_1, \ldots, k_n$ are sequentially chosen via the optimization
\begin{align}\label{eq:k-iteration}
  k_{q+1} = \argmax_{k} \left\| u_1(\cdot,k) - \sum_{j=1}^q v_j(k) u_1(\cdot, k_j) \right\|.
\end{align}
Computationally, the $\argmax$ is evaluated over a finite training set instead of a continuum. Since the above is an $\ell^2$-residual, in practice the solution to this greedy optimization problem on the finite training set is given by ordered pivots of a Cholesky or $Q R$ matrix factorization. (For the $Q R$ factorization, each column of the input matrix contains a snapshot.) Once the $k_q$ values have been computed, $n$ relatively expensive solution trajectories $u_2(t, k_q)$ and $u_3(t,k_q)$ are constructed, and the approximations
\begin{align*}
  \widehat{u}_2(t,k) &\coloneqq \sum_{q=1}^n u_2(t, k_q) v_q(k), & \widehat{u}_3(t,k) &\coloneqq \sum_{q=1}^n u_3(t, k_q) v_q(k)
\end{align*}
are built. Evaluation of $v_q$ at a fixed $k$ requires computation of the inexpensive model $u^L(t,k)$. The approximation above allows construction of $\widehat{u}^H(k)$ on the high-fidelity grid using computations on the low-fidelity grid for every value of $k$. We require only a one-time investment of $n$ solutions of $u^H$. When $n$ is small and $\widehat{u}^H$ is accurate, this can result in significant computational savings when analyzing the behavior of the family of solutions $u(\cdot, k)$ over the relevant range of $k$.

\subsection{Main contributions}

Our first contribution is the derivation of the error estimate
\begin{align}\label{eq:err-estimate}
  \left\| \mathcal{P}_n u(\cdot,k) - \widehat{u}_j(\cdot,k) \right\|_{\infty} &\lesssim C_1 h^p + C_2 (h/r^{j-1})^p,
\end{align}
where $p$ is the global truncation order of the discrete time integration method used to compute $u_j$, and $\mathcal{P}_n$ is a projection operator onto $\mathrm{span} \{ u(\cdot,k_q) \}_{q=1}^n$. The precise statement is given by Theorem \ref{thm:mf-error}. 

Our second contribution computationally effects convergence acceleration. The $j$-dependence in the error estimates above suggest that sequence transformation may be effective in accelerating convergence by eliminating the $j$-dependent error term. We would like to perform such an extrapolative transform at each instance of time, but the difficulty is that $u_2$ and $u_3$ ``live" on different grids. To rectify this situation, we perform spline approximations on each level, with the order of the spline matching $p$, the time integration order. The spline approximations then allow pointwise (in time) sequence transformation. We show that this strategy for convergence acceleration can be effective. The spline procedures and sequence transformation/extrapolation procedure is visually summarized in Figure~\ref{fig_schematic}. We observe in our examples that we can obtain $h^{p+1}$-order convergence in the accelerated solution despite the theoretical presence of $j$-independent terms in the estimate \eqref{eq:err-estimate}. This suggests that $C_1 \ll C_2$ can happen in practice.

\begin{figure}[h]
  \begin{center}
\includegraphics[width=0.8\textwidth]{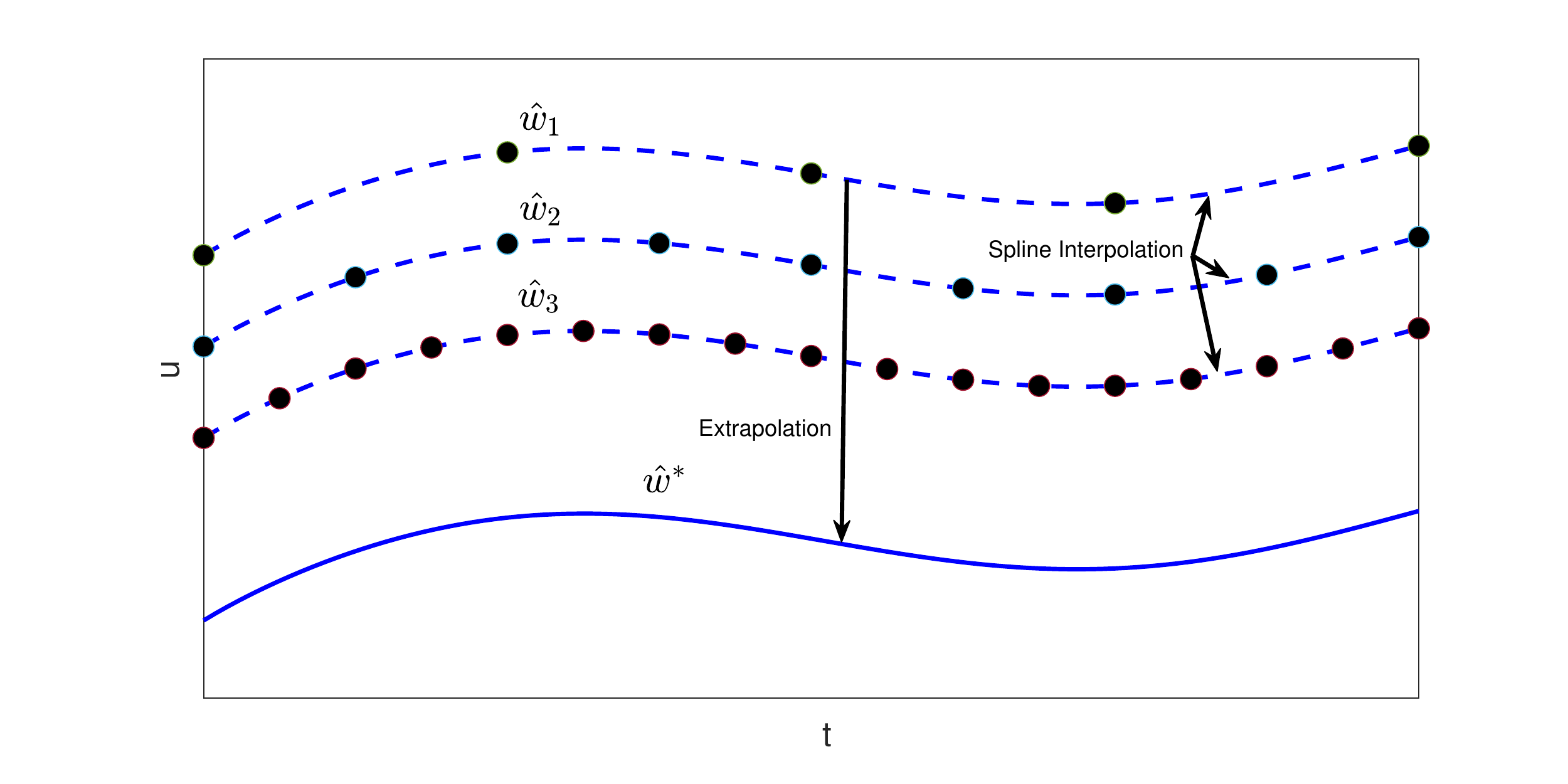}
\end{center}
\caption{\small{Schematic representation for convergence acceleration of time-dependent multifidelity models.}}\label{fig_schematic}
\end{figure}

The paper is organized as follows. Section \ref{sec2} introduces notation and the models of differing fidelities. Section \ref{sec3} describes the mathematical multifidelity procedure and contains our main error estimate, Theorem \ref{thm:mf-error}. Section~\ref{sec4} discusses convergence acceleration using spline interpolants and sequence transformations. Section \ref{sec:complexity} summarizes the entire algorithm. Finally, Section~\ref{sec:results} presents numerical examples for linear and nonlinear ODEs.

\section{Notation and setup}\label{sec2}


\subsection{Parameterized ODE solutions}

\begin{table}
  \begin{center}
  \resizebox{\textwidth}{!}{
    \renewcommand{\tabcolsep}{0.4cm}
    \renewcommand{\arraystretch}{1.3}
    {\scriptsize
    \begin{tabular}{@{}cp{0.8\textwidth}@{}}
      \toprule
      $t$, $T$ & Time variable $t$ taking values in $[0, T]$ \\
      $k$, $\mathcal{K}$ & Parameter value $k$ taking values in $\mathcal{K} \subset \R^d$ \\
      $M$ & Dimension of vector-valued solutions to an ODE \\
      $u(t,k)$ & $\R^M$-valued solution to a $k$-parameterized ODE at time $t$. The trajectory satisfies $u(\cdot,k) \in H$\\
      $H$ & Hilbert space $H$ containing solution trajectories, $u(\cdot, k) \in H$\\
      $h$, $N$ & Coarse timestep $h$, with $T = N h$\\
      $[N]$ & The set $\{0, 1, \ldots, N\}$ serving as indices for discrete times.\\
      $r$, $h_j$, $N_j$ & Integer $r \geq 2$ defining time step $h_j = h/r^{j-1}$ for ``level" $j$ approximation, using $N_j = N r^{j-1}$ equidistant time steps to reach $T$\\
      $H_j$ & Hilbert space containing $h_j$-discretized solution trajectories \\
      $u_j(i,k)$ & $\R^{M}$-valued discrete solution at time $t = i h_j$ computed using an integration method with timestep $h_j$, with $u_j(k) = u_j(\cdot,k) \in H_j$\\
      $p$, $P$ & Time integration global truncation error order $p$ and $(P+1)$-point Newton-Cotes quadrature rule \\
      $\mathcal{K}_n$ & Collection of $n$ points in $\mathcal{K}$ \\
      $\bs{G}$, $\bs{G}_j$ & $n \times n$ Gramian matrices formed from solution trajectories for $k \in \mathcal{K}_n$\\
      $\mathcal{V}$, $\mathcal{V}_j$ & Manifold of solutions for all $k \in \mathcal{K}$. Subsets of $H$ and $H_j$, respectively.\\
    \bottomrule
    \end{tabular}
  }
    \renewcommand{\arraystretch}{1}
    \renewcommand{\tabcolsep}{12pt}
  }
  \end{center}
  \caption{Notation used throughout this article.}\label{tab:notation}
\end{table}

We refer to Table \ref{tab:notation} for a summary of much of the notation in this article. The parameterized ODE \eqref{eq:ode}, where $u(t,k) \in \R^M$ depends on the parameters $k \in \mathcal{K} \subset \R^d$. We use $u^{(m)}$, $m = 1, \ldots, M$, to denote the components of $u$. We assume $\mathcal{K}$ is a compact set in $\R^d$ and that, given some terminal time $T > 0$, then the trajectory $u(\cdot,k)$ is smooth uniformly in $k$:
\begin{assumption}\label{assumption:core-assumption}
  The solution trajectories $u(\cdot, k)$ exist and are unique on $[0,T]$ for each $k \in \mathcal{K}$. Furthermore, the function $f(t,u)$ is smooth enough so that for some integer $p \geq 2$,
  \begin{align}\label{eq:Us-def}
    \max_{k \in \mathcal{K}} \max_{1 \leq m \leq M} \sup_{t \in [0,T]} \left| \frac{\partial^s}{\partial t^s} u^{(m)}(t,k)\right| \eqqcolon U^{(s)} &< \infty, & 0 \leq s & \leq p
  \end{align}
\end{assumption}
The above is relatively restrictive, requiring smoothness (up to order $p$) of the solution trajectories, with derivative bounds independent of $k$. 
The value of $p$ required is the convergence order of a time integration scheme.
The condition \eqref{eq:Us-def} allows us to conclude that the solution $u(\cdot, k): [0, T] \rightarrow \R^M$ to \eqref{eq:ode} is at least continuously differentiable on the compact interval $[0,T]$. Therefore, we have
\begin{align*}
  u(\cdot, k) &\in H & H \coloneqq L^2\left([0, T]; \R^M\right) &= \left\{ v: [0, T] \rightarrow \R^M \;\; \big| \;\; \left\| u \right\|_H < \infty \right\},
\end{align*}
with the inner product and norm
\begin{align*}
  \left\langle u, v \right\rangle &= \frac{1}{M T} \sum_{m=1}^M \int_0^T u^{(m)}(t) v^{(m)}(t) \dx{t}, & \| u\|^2 = \left\langle u, u \right\rangle
\end{align*}
where $u^{(m)}$ and $v^{(m)}$ are the components of the $M$-vectors $u$ and $v$, respectively. Note that we normalize the inner product by $1/(M T)$.

We are interested in computing approximations to the family of solutions
\begin{align*}
  \mathcal{V} = \left\{ u(\cdot, k) \;\;\big|\;\; k \in \mathcal{K} \right\} \subset H.
\end{align*}
More precisely, given $k \in \mathcal{K}$, we wish to construct an efficient and accurate approximation to the solution map $k \mapsto u(\cdot, k)$.

We require one additional assumption on the function $f$, namely that it is continuous in $k$.
\begin{assumption}\label{assumption:f-continuous}
  For every $t \in [0, T]$ and $u \in \R^M$, the function $f(t, u, k)$ is $k$-continuous for $k \in \mathcal{K}$. Also, the initial data $u_0(k)$ is continuous for each $k \in \mathcal{K}$.
\end{assumption}

\subsection{Time integration}\label{sec:time-integration}
We assume we have a stable and convergent numerical method to compute solutions to \eqref{eq:ode} for all $k \in \mathcal{K}$ over $t \in [0,T]$ that uses $N \in \N$ timesteps to reach $t = T$. (E.g., we assume $N$ is large enough for stability of explicit time integration methods uniformly in $k$.) Thus, define $h \coloneqq T /N$ as the timestep size, and let $h_j$, $j \in \N$, be a geometric sequence of timestep sizes
\begin{align}\label{eq:j-def}
  h_j &= \frac{h}{r^{j-1}}, & N_j \coloneqq N r^{j-1},
\end{align}
where $r \geq 2$ is an integer. We will use the solutions computed with timesteps $h_j$ (i.e., $N_j$ total timesteps) as our models of different fidelity. Suppose the rate of convergence of our time integration method is $p \geq 1$, and let $u_j(i,k) \in \R^M$ be the discrete solution at discrete time $i h_j$ ($i = 0, \ldots, N_j$) computed using the time integration method with a time step of $h_j$.

Let $[N] \coloneqq \{ 0, 1, \ldots, N\}$. An accurate time integration method produces vectors $u_j(i,k) \in \R^M$ satisfying
\begin{align*}
  u_j(i, k) &\approx u(i h_j, k), & i &\in [N_j], \hskip 8pt k \in \mathcal{K}
\end{align*}
We will primarily be interested in the values $j=1, 2$, and $3$, representing a three-level hierarchy of solutions.
We emphasize that, fixing $k$, the exact ODE solution $u(\cdot,k)$ is a function whose domain is the continuum $[0,T]$, but the discrete solution $u_j(\cdot, k)$ is a function whose domain is the finite set of indices $[N_j]$. The following is a standard estimate for the global truncation error committed by a time integration method of global order $p$ when applied to ODEs with smooth coefficients.
\begin{lemma}\label{lemma:global-truncation-error}
  Let $u_j^{(m)}(i, k)$, $k \in \mathcal{K}$, $i \in [N_j]$, be the $m$-th component at time index $i$ of the discrete solution computed using an order-$p$ time integration method with timestep $h_j$. Under Assumption \ref{assumption:core-assumption} with smoothness order $p$, then for $h_j$ small enough,
\begin{align*}
  \max_{m=1, \ldots, M} \max_{i = 0, \ldots, N r^{j-1}} \left| u^{(m)}(i h_j, k) - u^{(m)}_j(i, k) \right| \leq C(T,L) h_j^{p} = \frac{C(T,L)}{r^{j-1}} h^p,
\end{align*}
where $C(T,L)$ usually depends exponentially both on $T$ and bounds on derivatives of $f$.
\end{lemma}
Note that $C(T,L)$ does not depend on $k$ due to the assumption \eqref{eq:Us-def}.
We can define discretized Hilbert spaces $H_j$ that contain the discrete-time solutions for each $k \in \mathcal{K}$:
\begin{align}\label{eq:H-def}
  H_j &= \ell^2\left( [N_j]; \R^M \right), & \left\langle u, v \right\rangle_{j} &= \sum_{i=[N_j]} w_{j,i} v(i)^T u(i), &
  u(i) \in \R^M
\end{align}
where $w_j$ is a vector of $N_j+1$ weights. 
We assume the $w_j$ weight vectors are positive for each $j$, and that the entries sum to $1$ to reflect the $1/(M T)$ normalization in \eqref{eq:H-def}. We will make precise choices for these weights in the next section. The discrete solutions induce discretized versions of the compact manifold $\mathcal{V}$:
\begin{align*}
  \mathcal{V}_j = \left\{ u_j(\cdot,k) \;\; \big|\;\; k \in \mathcal{K} \right\} \subset H_j.
\end{align*}
The convergence in Lemma \ref{lemma:global-truncation-error} also implies that the discrete solution manifolds have bounded elements. In particular, since $\mathcal{K}$ is compact and $k\mapsto u(\cdot, k)$ is continuous by Assumption \ref{assumption:f-continuous}, we have
\begin{align}\label{eq:U-def}
  U_j = \max_{k \in \mathcal{K}} \| u_j(\cdot,k)\|_j < \infty \hskip 10pt \Longrightarrow \hskip 10pt U \coloneqq \max \left\{ \max_{k \in \mathcal{K}} \|u(\cdot,k)\|, \;\; \max_{j \geq 1} U_j \right\} < \infty,
\end{align}
where $\{U_j\}_{j \geq 1}$ is a positive and convergent (hence bounded) sequence by Lemma \ref{lemma:global-truncation-error}.

\subsection{Norms and inner products on $H_j$}
The discussion at the beginning of Section \ref{sec:time-integration} constructs the functions $u_j(k)$ so that they represent approximations to the exact solution trajectory $u(\cdot, k)$ evaluated on an equispaced grid. For our procedure, we require the ability to approximate inner products on $H$ using this discrete grid up to the order of accuracy $p$ of the time integration scheme. 
For this purpose, we turn to a composite Newton-Cotes quadrature rule.
A $(P+1)$-point closed Newton-Cotes rule on the interval $[a, b]$ has the form \begin{align}\label{eq:newton-cotes} \int_a^b f(x) \dx{x} &\approx \sum_{s=0}^{P} \widetilde{w}_s f(x_s), & x_s &= a + (b-a) \frac{s}{P}, \end{align}
with known, explicitly computable weights $\widetilde{w}_j$. For $P + 1 \leq 7$, the weights are all positive. For a function $g$ whose $(q-1)$th derivative is bounded on $[a,b]$, the rule has accuracy given by:
\begin{align}\label{eq:nc-error}
  \left| \int_a^b g(x) \dx{x} - \sum_{s=0}^P \widetilde{w}_s x_s \right| &= \mathcal{O}((b-a)^q), &q = q(P) &= 2 \left\lfloor \frac{P}{2} \right\rfloor + 3.
\end{align}
Now set $q = p+1$, and assume the order-$p$ smoothness as stated in Assumption \ref{assumption:core-assumption}. Choose $P(p)$ as
\begin{align}\label{eq:rp-relation}
  P(p) = \max \left\{ 2 \left\lfloor \frac{p-1}{2}\right\rfloor, 1 \right\}.
\end{align}
Our choice of $P(p)$ above is the smallest $P$ satisfying $q(P) = p+1$, and thus with this choice integrating a solution trajectory under Assumption \ref{assumption:core-assumption} achieves order of accuracy $p+1$ on individual subintervals of $[0,T]$ so that the composite rule has order-$p$ accuracy.

We can now define the weights $w_{j,i}$ defining the inner product on $H_j$. Assume that the number of coarse-level timesteps, $N$, is divisible by $P(p)$. Then for any $j \geq 1$ the interval $[0,T]$ can be divided up into $N_j/P$ subintervals, and a composite Newton-Cotes rule over $[0, T]$ acting on a function $v \in H$
\begin{align}\label{eq:composite-rule}
  Q_j[v] &\coloneqq \sum_{i=1}^{N_j/P} \sum_{s=0}^{P} \widetilde{w}_{j,i,s} v(((i-1)P + s + 1)h_j) \coloneqq \sum_{i \in [N_j]} w_{j,i} v_i, & \sum_{i \in [N_j]} w_{j,i} = 1 
\end{align}
where $\{\widetilde{w}_{j,i,s}\}_{s=0}^{P}$ are the weights $\{\widetilde{w}_s\}_{s=0}^{P}$ in \eqref{eq:newton-cotes} rescaled for the $i$th subinterval of $[0,T]$. The condition that the weights $w_{j,i}$ sum to $1$ is required for consistency of the $H_j$ discrete inner products with respect to the $H$ continuous inner product. 
Assuming $P + 1 \leq 7$, we use the vector of positive weights $w_j \in \R^{N_j+1}$ in this composite rule to define the norm and inner product on $H_j$ via the expression \eqref{eq:H-def}.

The case $P + 1 > 7$ only becomes relevant when we are using a time integration method with order $p$ equal to 7 or greater; this situation rarely happens in practice, so we hereafter assume $p \leq 6$ and $P + 1 \leq 7$. This quadrature rule has order of accuracy $p$ over the whole interval, which we codify below for the special case that we require.
\begin{lemma}[Composite Newton-Cotes accuracy]\label{lemma:nc-error}
  Let $u_{(j)}(\cdot,k) \in H_j$ be a sampling at the timesteps for level $j$ of an element in the solution $u(\cdot, k)$, i.e.,
  \begin{align}\label{eq:wj-def}
    u_{(j)}(i,k) &= u( i h_j,k ), & i &\in [N_j], \;\; k \in \mathcal{K}
  \end{align}
  Given $p \geq 1$, if $P$ is chosen as in \eqref{eq:rp-relation}, then under the conditions of Assumption \ref{assumption:core-assumption},
  \begin{align}\label{eq:nc-u-error}
    \left| \left\langle u(\cdot,k), u(\cdot,k') \right\rangle - \left\langle u_{(j)}(\cdot,k) u_{(j)}(\cdot,k') \right\rangle_j \right|
    &\leq C p^{p} h_j^p, & k, k' &\in \mathcal{K},
  \end{align}
  where $C$ is independent of $k$ and $k'$.
\end{lemma}
\begin{proof}
  Each individual Newton-Cotes rule spans a normalized interval of length $P h_j$ and is accurate to order $p+1$. Relation \eqref{eq:rp-relation} implies that $P \leq p-1$. Under Assumption \ref{assumption:core-assumption}, the integrand $u(\cdot,k) u(\cdot,k')$ has bounded derivatives of order $p = q-1$ on $[0,T]$. Thus for one component of the integrand, the Newton-Cotes rule commits an error scaling like,
  \begin{align*}
    \sum_{i=1}^{N_j/P} (P h_j)^{p+1}/(M T) = \frac{1}{T} N_j h_j (P h_j)^{p} \leq \frac{1}{M} p^p h_j^p,
  \end{align*}
  where we have used \eqref{eq:rp-relation} to conclude that $P(p) \leq \max\{ p-1, 1\} \leq p$, and the $1/(M T)$ factor arises because of the $1/(M T)$ normalization for the weights in \eqref{eq:composite-rule}. Summing over the $M$ components results in the estimate \eqref{eq:nc-u-error}. The constant $C$ appearing in the conclusion is a $k$-independent bound on the order-$p$ derivatives of the integrands, for which a loose bound is
  \begin{align*}
    \left| \frac{\partial^p}{\partial t^p} \left( u^{(m)}(t,k) u^{(m)}(t,k) \right) \right| \leftstackrel{\eqref{eq:Us-def}}{\leq} \sum_{r=0}^{p} \left(\begin{array}{c} p \\ r \end{array}\right) U^{(r)} U^{(p-r)}.
  \end{align*}
\end{proof}

\section{Time-dependent multifidelity approximations}\label{sec3}

The analytical result in this section is our first novel contribution: a proof that the multifidelity approximations $\widehat{u}_j$ (formally defined in this section) converge to an appropriate quantity as $h \rightarrow 0$. Our major result stating this is Theorem \ref{thm:mf-error}.

\subsection{Projection approximations}
Let $\mathcal{K}_n \subset \mathcal{K}$ be a set of $n \in \N$ points in parameter space:
\begin{align*}
  \mathcal{K}_n = \left\{ k_1, \ldots, k_n \right\} \subset \mathcal{K}.
\end{align*}
For a fixed $n$, we define subspaces spanned by $u(\cdot,\mathcal{K}_n)$ and $u_j(\cdot, \mathcal{K}_n)$,
\begin{align*}
  \mathcal{W}_n &= \mathrm{span}\left\{ u(k_1), \ldots, u(k_n) \right\} \subset H, &
  \mathcal{W}_{n,j} &= \mathrm{span}\left\{ u_j(k_1), \ldots, u_j(k_n) \right\} \subset H_j.
\end{align*}
We now define $\mathcal{P}_n$ and $\mathcal{P}_{n,j}$ as the orthogonal projectors onto $\mathcal{W}_n$ and $\mathcal{W}_{n,j}$, respectively:
\begin{align*}
  \mathcal{P}_n &: H \rightarrow \mathcal{W}_n, & \mathcal{P}_{n,j} &: H_j \rightarrow \mathcal{W}_{n,j}.
\end{align*}
We will show that the multifidelity approximation that we form converges to $P_n u$.
The approximation error committed by projecting the solution set $\mathcal{V}$ onto the subspace $\mathcal{W}_{n}$ is
\begin{align}\label{eq:e-def}
  e_{n}(\mathcal{V}) \coloneqq e\left(\mathcal{V}, \mathcal{W}_{n} \right) = \sup_{v \in \mathcal{V}} \left\| v - \mathcal{P}_{n} v \right\| = \sup_{k \in \mathcal{K}} \left\| u(k) - \mathcal{P}_{n} u(k) \right\|.
\end{align}
 The best possible error in approximating the true solution space $\mathcal{V}$ by an $n$-dimensional projection is,
\begin{align*}
  d_{n}(\mathcal{V}) = \inf_{\mathcal{W} \subset H,\; \dim W = n}\; e (\mathcal{V}, \mathcal{W}),
\end{align*}
The above is formulated on the continuous space $H$, which is not directly computable since the exact solutions $u(\cdot,k)$ are usually not available. Similar quantities can be defined to measure the error committed on the discrete level, e.g., the $H_j$-error committed by approximating $\mathcal{V}_j$ with $\mathcal{W}_{n,j}$.

On the discrete spaces $H_j$, one way to construct a sequence $k_1, k_2, \ldots$ for which $\mathcal{W}_{n,j}$ well-approximates $\mathcal{V}_j$ is by greedy procedure, in particular given $k_1, \ldots, k_n$, by picking $k_{n+1}$ as the $k$ value that maximizes a discrete version of the supremum argument in \eqref{eq:e-def}. Optimization on this discrete level can provide errors similar to optimizing over the continuous level. Below we cite a sufficient condition on the $j=1$ level.

\begin{lemma}[\cite{devore_greedy_2013}]\label{lemma:n-width}
  Suppose the parameter values $k_1, k_2, \ldots$ are chosen via the greedy procedure
  \begin{align}\label{eq:k-choice}
    k_{q+1} &= \argmax_{k \in \mathcal{K}} \left\| u_1(k) - \mathcal{P}_{q,1} u_1(k) \right\|_{1}, & q &\geq 0,
  \end{align}
  with $\mathcal{P}_{0,1}$ the zero operator. If there is a positive constant $\gamma > 0$ such that
  \begin{align}\label{eq:n-width-assumption}
    \frac{\max_{k \in \mathcal{K}} \| u_1(k) - \mathcal{P}_{q,1} u_1(k) \|_1}{\max_{k \in \mathcal{K}} \| u(k) - \mathcal{P}_{q} u(k) \|_1} \geq \gamma &> 0, & q &\geq 0
  \end{align}
  then,
  \begin{align*}
    e_{2 q}\left(\mathcal{V}\right) &\leq \frac{\sqrt{2}}{\gamma} \sqrt{d_q \left(\mathcal{V}\right)}, & q &\geq 1
  \end{align*}
\end{lemma}
\noindent
See also \cite{binev_convergence_2011} for related estimates. The condition \eqref{eq:n-width-assumption} ensures that the sequence $k_q$ is generated via a weak greedy algorithm. It is difficult in general to verify the assumption \eqref{eq:n-width-assumption}. However this is required for many computational model reduction methods that utilize snapshots, e.g., the reduced basis method \cite{Maday2002,patera_reduced_2007}, to prove convergence via the strategy in \cite{devore_greedy_2013}. Note that \eqref{eq:k-choice} is exactly the choice we make in \eqref{eq:k-iteration} for the multifidelity approximation. The optimization \eqref{eq:k-choice} above is stated as optimization over the continuum $\mathcal{K}$. In practice optimization is performed on a discretization of $\mathcal{K}$.

Our computations use the discrete projection operators $\mathcal{P}_{n,j}$, so our focus turns now to them. Given $k \in \mathcal{K}$, the projection $\mathcal{P}_{n,j} u_j(k)$ can be written as
\begin{align}\label{eq:ls-system}
  \mathcal{P}_{n,j} u_j(k) &= \sum_{q=1}^n v_{j,q}(k) u_j(k_q), & \bs{v}_j = \argmin_{\bs{w} \in \R^n} \left\| u_j(k) - \sum_{q=1}^n w_q u_j(k_q) \right\|_{j}
\end{align}
The unknown vector of coefficients $\bs{v}_j \in \R^n$ can be computed via the normal equations. To state this, we introduce the kernel functions for the continuous solutions $u(\cdot,k)$, and for the discrete solutions $u_j(\cdot,k)$, $j = 1, 2 \ldots, $
\begin{align*}
  K &: \mathcal{K} \times \mathcal{K} \rightarrow \R, & K_j &: \mathcal{K} \times \mathcal{K} \rightarrow \R \\
  K(k,k') &= \left\langle u(k), u(k') \right\rangle, & K_j(k,k') &= \left\langle u_j(k), u_j(k') \right\rangle_j.
\end{align*}
Then the normal equations formulation of \eqref{eq:ls-system} is
\begin{align}\label{eq:Gj-def}
  \bs{G}_j \bs{v}_j &= \bs{f}_j(k), & (G_j)_{p,q} &= K_j(k_q, k_p), & (f_j)_p(k) &= K_j(k,k_p),
\end{align}
for $p, q = 1, \ldots, n$. For future use, we similarly define the $n \times n$ matrix $\bs{G}$ and $n \times 1$ vector $\bs{f}(k)$ as containing inner products between the two exact solutions:
\begin{align}\label{eq:G-def}
  (G)_{p,q} &= K(k_q, k_p), & (f)_p(k) &= K(k, k_p)
\end{align}
The error relation \eqref{lemma:nc-error} allows us to conclude that evaluations of $K_j$ and $K$ are proximal.
\begin{lemma}\label{lemma:kernel-lemma}
  Under Assumption \ref{assumption:core-assumption}, then for any $k, k' \in \mathcal{K}$:
  \begin{align*}
    \sup_{k,k' \in \mathcal{K}} \left| K_j(k,k') - K(k,k') \right| &\leq C(T,L,U,p) h_j^p.
  \end{align*}
\end{lemma}
\begin{proof}
  Let $u_{(j)}(\cdot,k) \in H_j$ be the $h_j$-sampling of the solution $u(\cdot, k)$ as defined in \eqref{eq:wj-def}. 
  Then define the new kernel function
  \begin{align*}
    K_{(j)}(k, k') = \left\langle u_{(j)}(k), u_{(j)}(k') \right\rangle_j
  \end{align*}
  We have,
  \begin{align}\label{eq:lemma-temp1}
    \left| K(k,k') - K_j(k,k') \right| &\leq \left| K(k,k') - K_{(j)}(k,k') \right| + \left| K_{(j)}(k,k') - K_j(k,k') \right|
  \end{align}
  We show that each term on the right-hand side above scales like $h_j^p$. We can immediately bound the first term,
  \begin{align}\label{eq:lemma-temp3}
    \left| K(k,k') - K_{(j)}(k,k') \right| \;\;\;\;\; \leftstackrel{\textrm{Lemma \ref{lemma:nc-error}}}{\leq} C(p) h_j^p.
  \end{align}
  For the second term, define
  \begin{align*}
    e_j(k) \coloneqq u_{(j)}(k) - u_j(k) \in H_j.
  \end{align*}
  By Lemma \ref{lemma:global-truncation-error} and the weight summation condition \eqref{eq:composite-rule},
  \begin{align}\label{eq:ej-bound}
    \left\| e_j(k) \right\|_j \leq C h_j^p.
  \end{align}
  Then the triangle and Cauchy-Schwarz inequalities yield,
  \begin{align}
    \nonumber\left| K_{(j)}(k,k') - K_j(k,k') \right| &\leq \left| \left\langle u_{(j)}(k), e_j(k') \right\rangle_j \right| + \left| \left\langle e_j(k), u_j(k') \right\rangle_j \right| \\
    \label{eq:lemma-temp2} &\leftstackrel{\eqref{eq:ej-bound},\eqref{eq:U-def}}{\leq} C_1 U h_j^p + C_2 U h_j^p \leq C_3(U,T,L) h_j^p
  \end{align}
  Using \eqref{eq:lemma-temp2} and \eqref{eq:lemma-temp3} in \eqref{eq:lemma-temp1} proves the result.
\end{proof}

\subsection{Multifidelity approximations}
Our multifidelity approximations $\widehat{u}_j$ are defined by the linear least-squares solution $\bs{v}_1$ to the $j=1$ version of \eqref{eq:ls-system}. These approximations are, respectively,
\begin{align}\label{eq:multifidelity-approximations}
  \widehat{u}_j(\cdot,k) &= \sum_{q=1}^n v_{1,q}(k) u_j(\cdot,k_q), & j &\geq 1
\end{align}
The ``ideal" function that $\widehat{u}_j(k)$ represents is the $H$-projection approximation $P_{n} u(\cdot,k)$. The pointwise proximity of these two functions on the $h_j$ grid is of the order $h^p$ uniformly in $k$.
\begin{theorem}\label{thm:mf-error}
  Fix $n \in \N$, and assume that $\sigma_{\mathrm{min}}(\bs{G}) > 0$, where $\sigma_{\mathrm{min}}$ is the smallest singular value of $\bs{G}$. Let $w(\cdot,k) \coloneqq P_n u(\cdot, k) \in H$. Then there exists $\widebar{h} > 0$ such that for all $h < \widebar{h}$,
  \begin{align}\label{eq:mf-error}
    \max_{k \in \mathcal{K}}~\max_{m=1, \ldots, M}~\max_{i \in [N_j]} \left| w^{(m)}(i h_j, k) - \widehat{u}^{(m)}_j(i, k) \right| \leq C_1 h^p + C_2 h_j^p,
  \end{align}
  where
  \begin{align*}
    C_1(T, L, p, U, n, \bs{G}) &= C(T,L,p) \frac{U^3 n^2}{\sigma^2_{\mathrm{min}}(\bs{G})}, &
    C_2(T, L, p, U, n, \bs{G}) &= C(T,L,p) \frac{U^2 n}{\sigma_{\mathrm{min}}(\bs{G)}}
  \end{align*}
\end{theorem}
\begin{proof}
  By Lemma \ref{lemma:kernel-lemma}, then $\left\| \bs{G}_1 - \bs{G} \right\|\rightarrow 0$ as $h \downarrow 0$. Choose $\widebar{h}$ such that
  \begin{align}\label{eq:sv-bounds}
    \left\| \bs{G}_1 - \bs{G} \right\| < \frac{1}{2} \sigma_{\mathrm{min}}(\bs{G}),
  \end{align}
  for all $h < \widebar{h}$. Define the vectors $\bs{v}$ and $\bs{v}_1$ as solutions to the systems
  \begin{align*}
    \bs{G} \bs{v} &= \bs{f}(k), & \bs{G}_1 \bs{v}_1 = \bs{f}_1(k),
  \end{align*}
  where the vectors $\bs{f}_1$, $\bs{f}$, and matrices $\bs{G}$, $\bs{G}_1$ are defined in \eqref{eq:Gj-def} and \eqref{eq:G-def}. Both $\bs{G}$ and $\bs{G}_1$ are invertible since $\sigma_{\mathrm{min}}(\bs{G}) > 0$ and due to \eqref{eq:sv-bounds}. Now note that
  \begin{align}
    \nonumber \left| w^{(m)}( i h_j, k) - \widehat{u}^{(m)}_j(i,k) \right| =& \left|\sum_{q=1}^n  v_q(k) u^{(m)}(i h_j,k_q) - \sum_{q=1}^n v_{1,q}(k) u_j^{(m)}(i, k_q)\right|  \\
                                                                 \nonumber \leq& \left| \sum_{q=1}^n \left( v_q(k) - v_{1,q}(k) \right) u^{(m)}\left(i h_j, k_q\right) \right| + \\
                                                                 \nonumber & \left| \sum_{q=1}^n v_{1,q}(k) \left( u^{(m)}(i h_j, k_q) - u_j^{(m)}(i, k_q) \right) \right| \\
    \label{eq:thm-temp1}\leq & U \sqrt{n} \left\| \bs{v} - \bs{v}_1 \right\| + C h_j^p \sqrt{n} \left\| \bs{v}_1 \right\|,
  \end{align}
  where the last inequality uses Cauchy-Schwarz, and in this proof we use $\|\cdot\|$ to denote the standard Euclidean 2-norm on boldface vectors. Therefore, we need only show that $\left\| \bs{v} - \bs{v}_1 \right\|$ is on the order $h^p$ and that $\|\bs{v}_1\|$ is bounded. We have that
  \begin{align}\label{eq:thm-tempv}
    \left\| \bs{v}_1 \right\| \leq \left\| \bs{G}_1^{-1} \right\| \left\| \bs{f}_1 \right\| \leq \frac{\sqrt{n}}{\sigma_{\mathrm{min}}(\bs{G}_1)} U^2 \leq \frac{2 \sqrt{n}}{\sigma_{\mathrm{min}}(\bs{G})} U^2,
  \end{align}
  where the last inequality holds since by \eqref{eq:sv-bounds},
  \begin{align*}
    \left\| \bs{G}_1 - \bs{G} \right\| < \frac{1}{2} \sigma_{\mathrm{min}}(\bs{G}) \hskip 5pt \Longrightarrow \hskip 5pt \frac{1}{\sigma_{\mathrm{min}}(\bs{G}_1)} < \frac{2}{\sigma_{\mathrm{min}}(\bs{G})}.
  \end{align*}
  For the second term in \eqref{eq:thm-temp1}, Lemma \ref{lemma:kernel-lemma} implies that
  \begin{align}\label{eq:entrywise-proximity}
    \left| (G)_{p,q} - (G_j)_{p,q} \right| &\leq C h_j^p, & \left| (f)_p - (f_j)_p \right| & \leq C h_j^p, & p, q &= 1, \ldots, n.
  \end{align}
  I.e., the vector $\bs{v}_1$ is the solution to a perturbed version of the linear system $\bs{G} \bs{v} = \bs{f}(k)$. We now use a standard result in linear algebra: if $\bs{G}$ is square and invertible, then
  \begin{align*}
    \left\| \bs{v} - \bs{v}_1 \right\| \leq \frac{1}{\eta} \left( \left\| \bs{f}(k) - \bs{f}_1(k) \right\| + \left\| \bs{G} - \bs{G}_1 \right\| \left\| \bs{v} \right\| \right),
  \end{align*}
  where $\eta$ satisfies
  \begin{align*}
    \eta \coloneqq \sigma_{\mathrm{min}}(\bs{G}) - \left\| \bs{G}_1 - \bs{G} \right\| \leftstackrel{\eqref{eq:sv-bounds}}{>} \frac{1}{2} \sigma_{\mathrm{min}}(\bs{G}) > 0.
  \end{align*}
  The entrywise proximity relations \eqref{eq:entrywise-proximity} imply that
  \begin{align*}
    \left\| \bs{G} - \bs{G}_1 \right\| &\leq n C h^p, & \left\| \bs{f} - \bs{f}_1 \right\| &\leq \sqrt{n} C h^p.
  \end{align*}
  We therefore have
  \begin{align}
    \nonumber\left\| \bs{v} - \bs{v}_1 \right\| &\leq \frac{1}{\eta} \left( C \sqrt{n} h^p + C n h^p \|\bs{v}\| \right)  \\
    \label{eq:thm-temp3}&\leq \frac{C h^p}{\eta} \left( \sqrt{n} + n \|\bs{G}^{-1}\| \|\bs{f}\|\right) \leq \frac{C h^p}{\sigma_{\mathrm{min}}(\bs{G})} \left( \sqrt{n} +  \frac{n^{3/2}}{\sigma_{\mathrm{min}}(\bs{G})} U^2 \right).
  \end{align}
  Using \eqref{eq:thm-tempv} and \eqref{eq:thm-temp3} in \eqref{eq:thm-temp1} yields the result. 
\end{proof}
The assumption $\mathrm{\sigma}_{\mathrm{min}}(\bs{G}) > 0$ is equivalent to assuming that the set of $n$ solutions $u(\cdot, \mathcal{K}_n)$ is linearly independent in $H$. The appearance of $1/\sigma_{\mathrm{min}}(\bs{G})$ in \eqref{eq:mf-error} is expected due to worst-case linear system perturbation theory, but since the bound for this term is loose we expect the estimate \eqref{eq:mf-error} to be pessimistic in magnitude. The dependence of $C_1$ on $\sigma^2_{\mathrm{min}}(\bs{G})$ is another worst-case estimate, and is sharp only when $\bs{f}_1(k)$ has large component pointing in the minimal singular direction of $\bs{G}_1$. 



\begin{remark}\label{rem:mf}
  The error in \eqref{eq:mf-error} is the sum of two terms: One term is independent of $j$, and another scales like $h_j^p$. Such an error behavior suggests that we may be able to accelerate convergence to reduce the $h_j^p$ error term by usage of Richardson extrapolation. However, the estimate \eqref{eq:mf-error} suggests that the right-hand side is dominated by the $j$-independent $h^p$ term. A Richardson Extrapolation technique operating on different $j$ levels cannot eliminate this term, and an extrapolated approximation will have error behaving still like $h^p$. For the numerical results we have tested, the $h_j^p$ term appears to dominate the error behavior and so extrapolation techniques are successful. The observed $h_j^p$ dependence may result either from a lack of sharpness of our estimate, or is possibly the result of the particular examples we show and does not hold in general.
\end{remark}
The theorem above relates the error of $\widehat{u}_j$ to $P_n u$. If the assumption of Lemma \ref{lemma:n-width} holds, then we in addition have that the error between $P_n u$ and $u$ is comparable to $\sqrt{d_{n/2}(\mathcal{V})}$.

\section{Convergence acceleration}\label{sec4}
We have discussed computation of the multifidelity approximation $\widehat{u}_j(k)$, which is a member of $H_j$. The goal of this section is to illustrate that convergence of this approximation can be accelerated if we have knowledge of $\widehat{u}_1(k)$, $\widehat{u}_2(k)$, and $\widehat{u}_3(k)$. The cost of obtaining these three solutions (essentially just the cost of $u_1(k)$) is much less than the cost of computing the three solutions $u_1(k)$, $u_2(k)$, and $u_3(k)$ so that the multifidelity procedure can significantly speed up sequence transformation procedures.

\subsection{Connection operators: splines}
The multifidelity approximation $\widehat{u}_j$ that we have constructed lives in $H_j$. We desire a method to ``lift" this to $H$. Because of our smoothness assumptions, we turn to B-splines to accomplish this. The multifidelity reconstruction $\widehat{u}^{(m)}_j(\cdot,k)$ is a vector in $\R^{N_j+1}$ with data associated to time instances
\begin{align*}
  t_{j,i} &= i h_j, & i &\in [N_j].
\end{align*}
For a fixed multifidelity level $j$, and fixed time-stepping order of accuracy $p > 0$, we define a knot vector $\xi_i$ for use in spline construction. The first $p$ knots coincide, followed by equispaced knots, followed again by coincident knot values:
\begin{align}\label{eq:knots}
  \xi_i &= \left\{ \begin{array}{rrrl} 0, & 0 \leq &i& \leq p \\
                                    \frac{i-p}{N_j-p}, & p+1 \leq &i& \leq N_j+p+1 \\
                                    1, & N_j + p + 2 \leq &i& \leq N_j + 2p + 2.
  \end{array}\right.
\end{align}
In one dimension, basis splines (B-splines) are defined recursively using a knot vector, starting with piecewise constants
\begin{align*}
B_{i,0} (\xi) = \begin{cases}
1 \qquad \textrm{if} \quad  \xi_i \leq \xi < \xi_{i+1}\\
0 \qquad \textrm{otherwise}
\end{cases}
\end{align*}
for $p=0$ and
\begin{align}\label{bspline1}
B_{i,p} (\xi) = \frac{\xi - \xi_i}{\xi_{i+p} - \xi_i} B_{i,p-1}(\xi) +  \frac{\xi_{i+p+1} - \xi}{\xi_{i+p+1} - \xi_{i+1}} B_{i+1,p-1}(\xi)
\end{align}
for $p>0$. We choose the uniform knots \eqref{eq:knots} because our data lies on a uniform time grid.
Then given a fixed $k \in \mathcal{K}$, we can use the discrete-time approximation $\widehat{u}_j(k)$ as data in a global B-splines approximation problem:
\begin{align}\label{spline-coef}
  \widehat{u}^{(m)}_j(i, k) &= \displaystyle \sum_{s \in [N_j]}  \alpha_i B_{s,p}(t_{j,i}/T), & \alpha_0 &=  \widehat{u}^{(m)}(t_0,k), & \alpha_{N_j} &= \widehat{u}^{(m)}(t_{N_j},k),
\end{align}
for $1 \leq i \leq N_j-1$. The above system represents $N_j+1$ equations in $N_j+1$ unknowns $\alpha_i$, which can be solved. Once the $\alpha_i$ coefficients are known, we can form the spline approximation
\begin{align}\label{eq:spline-approximation}
  \widehat{w}^{(m)}_j(t,k) &= \sum_{i \in [N_j]}^n \alpha_i B_{i,p}(t/T), & 0 &\leq t \leq T.
\end{align}
Repeating this for $m = 1, \ldots, M$, we can create a continuous approximation $\widehat{w}_j(\cdot,k) \in H$ to the discrete multifidelity data $\widehat{u}_j(\cdot,k) \in H_j$ at any time value $t \in [0,T]$, and this approximation is accurate to order $p$.

\subsection{Richardson Extrapolation}
With $w = P_n u$, Theorem \ref{thm:mf-error} states that the multifidelity approximations $\widehat{u}_j$ satisfy
\begin{align}\label{eq:uj-estimate}
  \widehat{u}_j(i,k) &\simeq w(i h_j, k) + C_1 h^p + \frac{C_2}{r^{p(j-1)}} h^p,
\end{align}
This suggests that the use of sequence transformations operating on the index $j$ to accelerate convergence may be effective \cite{STOER1966,WENIGER1989189}. We adopt Richardson extrapolation in particular, though there are many other extrapolation procedures \cite{Joyce71,Brezinski96,Brezinski2000}. Throughout this section, we use $x_j$ to denote a generic scalar in the spline-postprocessed multifidelity sequence. I.e., for some fixed $t \in [0, T]$, $k \in \mathcal{K}$, and $m = 1, \ldots, M$, we let
\begin{align*}
  x^{(m)}_j \coloneqq \widehat{w}^{(m)}_j(t, k) \approx w^{(m)}(t, k) + C_1 h^p + C_1 r^{(1-j)p} h^p.
\end{align*}
Although we know the convergence order $p$ from an \textit{a priori} understanding of the time integration method, we use three levels to estimate this order. The Richardson extrapolation formula for estimating the order and the resulting extrapolation is, respectively,
\begin{align}\label{eq:p-estimation}
  p^\ast &= \log_r \left( \frac{x_1 - x_2}{x_2 - x_3} \right), & x^\ast &= \frac{r^{p^\ast} x_{3} - x_2}{r^{p^\ast} - 1} = c^\ast x_3 + (1-c^\ast) x_2,
\end{align}
where we have defined $c^\ast \coloneqq r^{p^\ast}/(r^{p^\ast}-1)$. With this order of accuracy in hand, the accelerated estimate $x^\ast$ is our final computed solution. For arbitrary $t \in [0,T]$, $k \in \mathcal{K}$, and $m = 1, \ldots, M$, the above procedure can be repeated on the sequence $x_j$ defined below, producing the estimate $x^\ast$:
\begin{align}\label{eq:sequence-transformation}
  x_j &= \widehat{w}^{(m)}_j(t, k), & \widehat{w}^{\ast(m)}(t,k) \coloneqq x^\ast \approx w^{(m)}(t, k).
\end{align}
The final output of our algorithm is $\widehat{w}^\ast(\cdot,k) \in H$. Note that, on account of the behavior \eqref{eq:uj-estimate}, we only expect $x^\ast$ to approximate $w$ with an error of $h^p$. However, we will see in our numerical results section that accuracy of order $p+1$ is observed.

\subsection{Computational Complexity of the Trifidelity Construction}\label{sec:complexity}
This section summarizes the entire algorithm. The simulation cost for the convergence-accelerated trifidelity algorithm is divided in two main parts:
\begin{itemize}
  \item Offline Stage -- Identification of $\mathcal{K}_n$ and computation of $u_2(\cdot,\mathcal{K}_n)$ and $u_3\left(\cdot,\mathcal{K}_n\right)$.
  \begin{itemize}
    \item Low fidelity computations: compute $Q \gg 1$ simulations of $u_1$
    \item Important sample selection: Approximate the $j=1$ optimization in \eqref{eq:k-choice} using the $Q$ simulations of $u_1$ via a pivoted Cholesky decomposition. This requires $\mathcal{O}(Q n^2)$ operations. The pivots identify $\mathcal{K}_n = \{k_1, \ldots, k_n\} \subset \mathcal{K}$.
    \item Medium and high fidelity computations: $n$ simulations of $u_2$ and $u_3$ at parameter locations $\mathcal{K}_n$.
  \end{itemize}
  \item Online Stage -- given $k$, compute $\widehat{w}^\ast(t,k)$, an approximation to $u(t,k)$.
  \begin{itemize}
    \item Low fidelity computation: Evaluate $u_1$ at parameter value $k$
    \item Multifidelity interpolation operator: compute medium- and high-fidelity approximations $\widehat{u}_2$ and $\widehat{u}_3$ in \eqref{eq:multifidelity-approximations}, involving the solution to the $j=1$ version of \eqref{eq:Gj-def}. The cost of this construction is dominated by the cost of a least-squares solve on the low-fidelity mesh, and hence requires $\mathcal{O}(N_1 n^2)$ operations.
    \item Spline interpolation: For $j = 1, 2, 3$, solve \eqref{spline-coef} to obtain the spline approximation \eqref{eq:spline-approximation} to obtain $\widehat{w}_j$.
    \item Sequence transformation: use \eqref{eq:p-estimation} and \eqref{eq:sequence-transformation} to compute the estimator $\widehat{w}^\ast$ at parameter value $k$, which can be evaluated at any $t \in [0, T]$.
  \end{itemize}
\end{itemize}
The major computational burden is only in the computation of $n$ medium and high fidelity solutions, which is a one-time (``offline") cost. Once this cost has been invested, one may compute the accelerated multifidelity estimator $\widehat{w}$ at the cost of only the low-fidelity model $u_1$.

\section{Numerical results}\label{sec:results}
We use this section to demonstrate the effectiveness of the accelerated multifidelity procedure. We wish to illustrate that one can use quite general time integrators. To this end, we will use the standard second-, third-, and fourth-order Runge-Kutta (RK, multi-stage) and Adams-Bashforth (AB, multi-step) schemes. 
The Runge-Kutta  schemes RK2, RK3, and RK4 are second-, third-, and fourth-order globally accurate, respectively, and similarly for the Adams Bashforth schemes, which we denote AB2, AB3, and AB4.

\subsection{Damped harmonic oscillator}
In this section we consider a second-order linear parameterized ODE and demonstrate different steps of the convergence acceleration algorithm on this illustrative example. The linear ODE model of a particular unforced mass-spring-damper is
\begin{equation}\label{linear_ODE}
  \ddot{u} + (0.1 + k/100) \dot{u} + k u = 0, \quad u(0)=1,~ \dot{u}(0)=10,
\end{equation}
where $k \in [5, 25]$ is the parameter, and hence the stiffness coefficient and the damping coefficient for the system are dependent parameters.

\subsubsection{Multifidelity approximations}
We construct three different approximations, $u_1, u_2$ and $u_3$ associated with time step values $h_1=0.1, h_2=0.05$ and $h_3=0.025$, respectively, and run the multifidelity procedure to construct a numerical approximation to the solution $u$ for arbitrary $k \in [5, 25]$. Different solution realizations obtained with the solver $RK4$ on the low fidelity model are shown in Figure~\ref{fig_1}. 
\begin{figure}[h]
\centering
\includegraphics[width=0.8\textwidth]{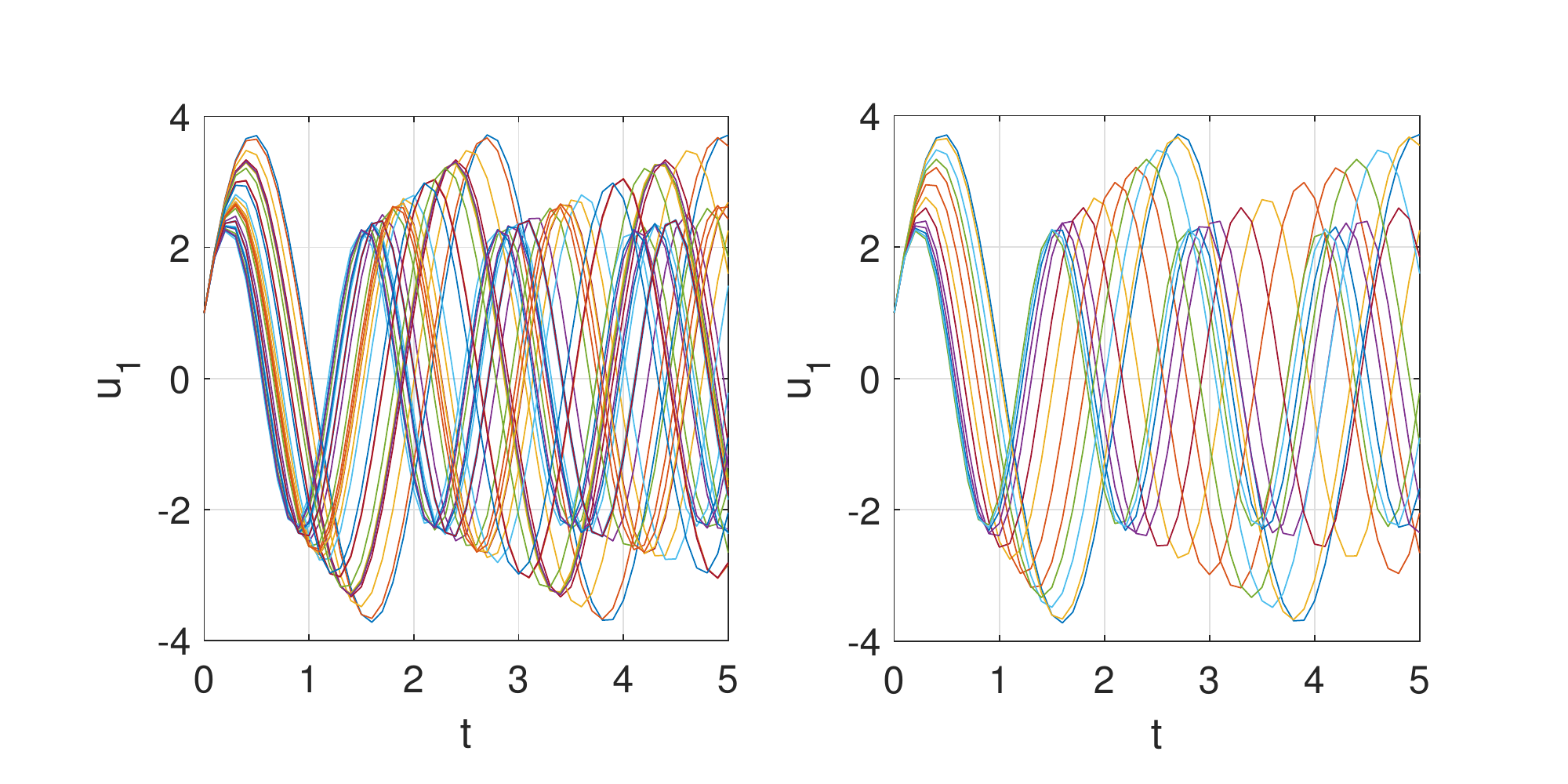}
  \caption{\small{Ensemble of solution realizations of the ODE \eqref{linear_ODE} using a low-fidelity RK4 solver. Left: an ensemble of low-fidelity trajectories. Right: $n=13$ trajectories identified via a pivoted Cholesky decomposition.}}\label{fig_1}
\end{figure}
The first step of the multifidelity procedure is to collect solution trajectories $u_1$ for many values of $k$. We choose $100$ values of $k\in [5,25]$ via Monte Carlo Sampling. We use this size-$100$ grid as a proxy for the continuum in the optimization problem \eqref{eq:k-choice} (i.e., $Q$ from Section \ref{sec:complexity} is set to 100); this results in $n=13$ parameter values $k_1, \ldots, k_n$ along with medium- and high- fidelity solutions computed on these parameter values.

\subsubsection{Sequence transformation and acceleration}
We investigate the convergence order of $\widehat{u}_3(k)$ at $t = 2.5$, for the two values $k=11$ and $k = 16$. The convergence order $p^\ast$ is estimated via \eqref{eq:p-estimation}, with $x_j$ being reconstructed multifidelity solutions $\widehat{u}_j(k)$ at the fixed time instance $t=2.5$, and this $p^\ast$ is used for all time $t$. The computed values of $p^\ast$ for particular parameter values $k$ are given in Table \ref{tabNE2_30}. The convergence order $p^\ast$ mirrors the order of the convergence $p$ of the time-stepping algorithm, regardless of whether a multi-stage (Runge-Kutta) or multistep (Adams-Bashforth) scheme is used.

\begin{table}[!h]
\caption{Convergence Order Estimation }
\resizebox{\textwidth}{!}{
\centering
\begin{tabular}{|c |c c| c c| c c| c c| c c| c c| }
\hline\hline
\textrm{Time-Stepping Method}   & \multicolumn{2}{ c| }{RK2} & \multicolumn{2}{ c| }{AB2} & \multicolumn{2}{ c| }{RK3} & \multicolumn{2}{ c| }{AB3} & \multicolumn{2}{ c| }{RK4} & \multicolumn{2}{ c| }{AB4}   \\
\hline
$k$   & 11 & 16 & 11  & 16 & 11 & 16  & 11 & 16 & 11 & 16 & 11 & 16   \\
$p^\ast$  & 2.20 & 1.97 & 2.19 & 2.04 & 3.01 & 3.29  & 2.97 & 3.36 & 4.38 & 3.95 & 4.31 & 3.85    \\
\hline
\end{tabular}
}
\label{tabNE2_30}
\end{table}

Once the multifidelity solutions $\widehat{u}_2$ and $\widehat{u}u_3$ solutions are built, these solutions are interpolated with an order-$p$ spline, where $p$ is again the order of the time integration method. Figure~\ref{fig_3} shows the spline curve $\widehat{w}_2$ and $\widehat{w}_3$ computed from the multifidelity data $\widehat{u}_2$ and $\widehat{u}_3$, respectively. For better resolution only results for $t \in [0,1]$ are visualized.

\begin{figure}[h]
\centering
\includegraphics[width=0.9\textwidth]{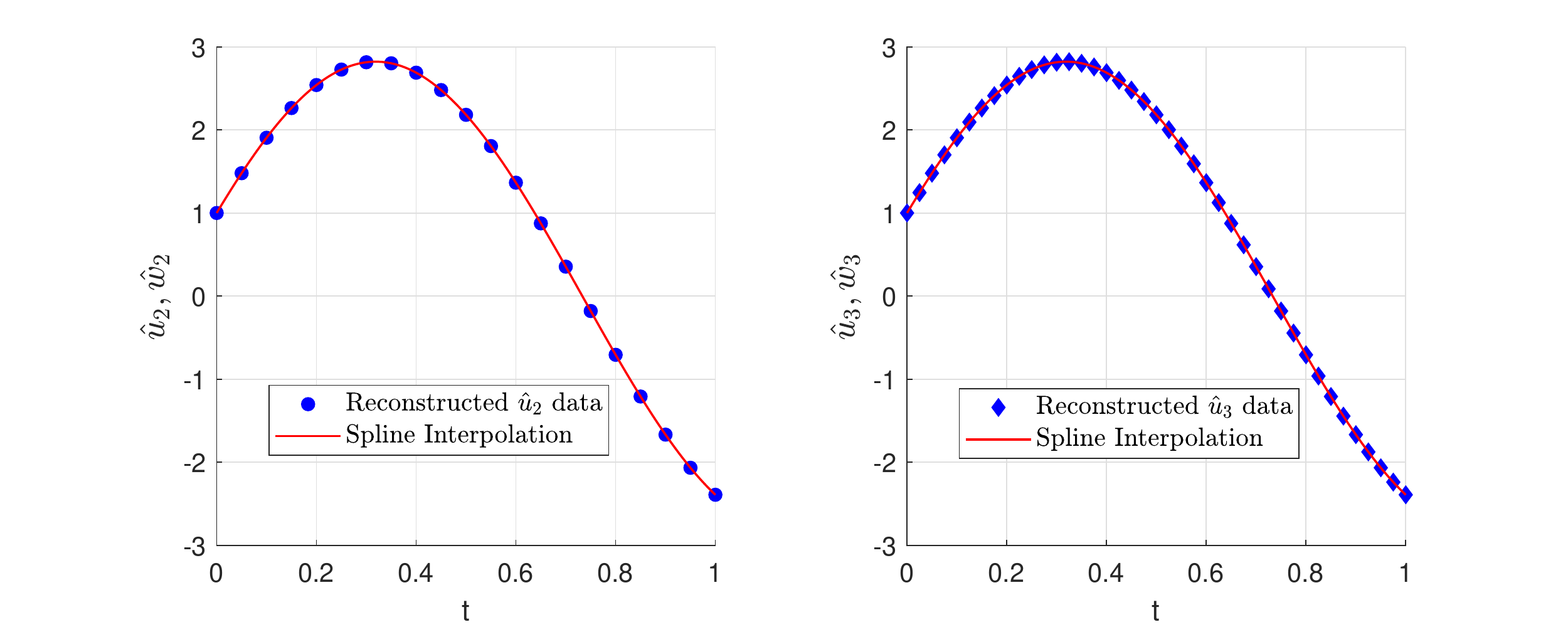}
\caption{\small{Interpolation with high order spline for reconstructed $\hat{u}_2$ and $\hat{u}_3$ solutions. Note that we do not use $\hat{u}_1$ for the convergence acceleration.}}\label{fig_3}
\end{figure}

For a given parameter $k$ and time instance $t$, the spline-reconstructed medium and high fidelity solutions $\hat{w}_2$ and $\hat{w}_3$ are used to obtain the extrapolated solution $\widehat{w}^\ast$ via
\begin{align}\label{eq:richardson-extrapolation}
  \widehat{w}^\ast = c \widehat{w}_3 - (c-1) \widehat{w}_2
\end{align}
This equation is equivalent to Equation~\eqref{eq:p-estimation} with $c=r^p/(r^p-1)$. In this example, $r=2$ since the timestep is halved between fidelities ($2 h_{j+1} = h_j$). We can explicitly compute the Richardson Extrapolation weights for $p=2,3,4$:
\begin{align*}
  p^\ast = 2 &\Longrightarrow c^\ast = \frac{4}{3} \approx 1.33, & p^\ast = 3 &\Longrightarrow c^\ast = \frac{8}{7} \approx 1.14, & p^\ast = 4 &\Longrightarrow c^\ast = \frac{16}{15} \approx 1.06.
\end{align*}

We test the accuracy of this approach for different values of $c$; based on our convergence theory, if our spline approximation is of the appropriate order, then we expect that $c = c^\ast$ will produce the best results (lowest error). We can confirm this behavior in Figures \ref{fig_4} and \ref{fig_5}. The relative error is shown for solvers of different convergence orders ($RK2$, $RK3$, and $RK4$ for Figure \ref{fig_4}, and $AB2$, $AB3$, and $AB4$ for Figure \ref{fig_5}), and different orders of spline interpolation are used. Relative error is measured in the normalized $\ell^2$ norm of the vector of values on a fine grid with stepsize $h = 10^{-3}$.

When the spline order of convergence is greater than or equal to the order of the convergence of the time-stepping algorithm, we expect $c = c^\ast$ to produce the best error. This expectation is realized in these Figures: increasing the spline order to the time-stepping order $p$ causes the minimum error to happen at $c = c^\ast$; increasing the spline order beyond this produces no change since the error is then limited by the time-stepping scheme.

We emphasize that these experiments are testing more than simply ``standard" Richardson Extrapolation: they are also verifying the convergence rate of the multifidelity approximation given in Theorem \ref{thm:mf-error}. The difference between a ``standard" Richardson Extrapolation technique and this multifidelity technique is that the ``standard" approach requires solutions $u_2$ and $u_3$, which are relatively expensive. In contrast, the multifidelity procedure requires only the surrogates $\widehat{u}_2$ and $\widehat{u}_3$, which can be generated at the cost of the much more inexpensive model $\widehat{u}_1$. (After some offline work has been invested, see Section \ref{sec:complexity}.)

\begin{figure}[h]
\includegraphics[width=0.9\textwidth]{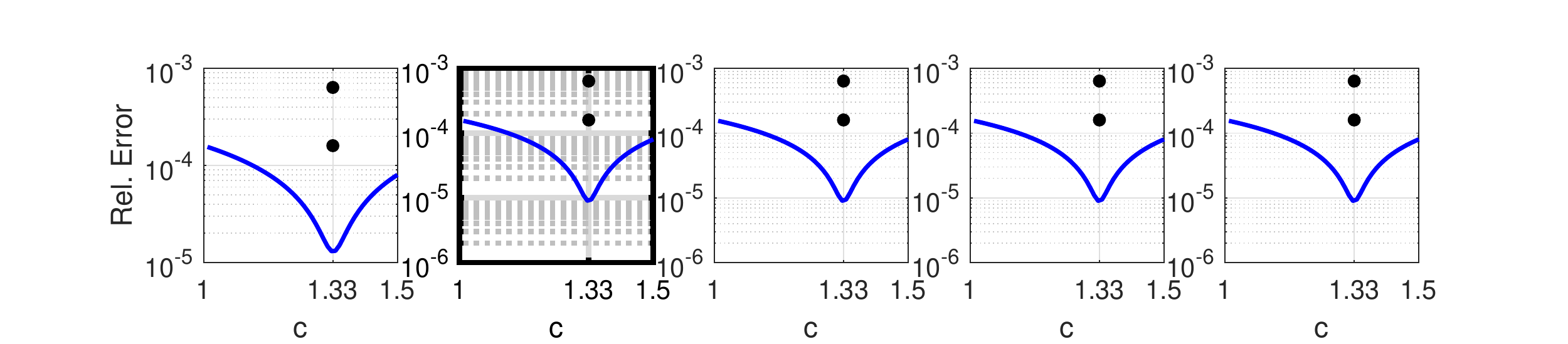}
\includegraphics[width=0.9\textwidth]{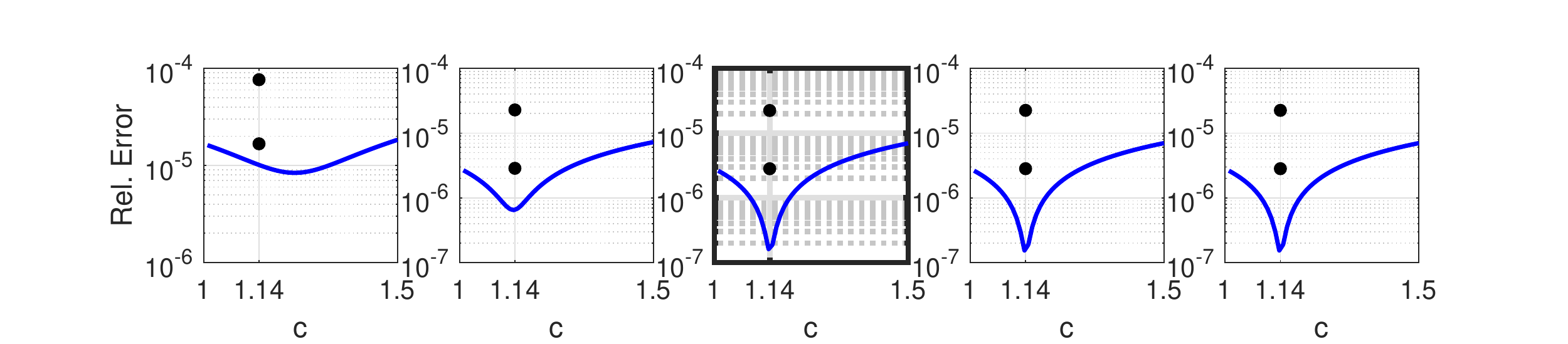}
\includegraphics[width=0.9\textwidth]{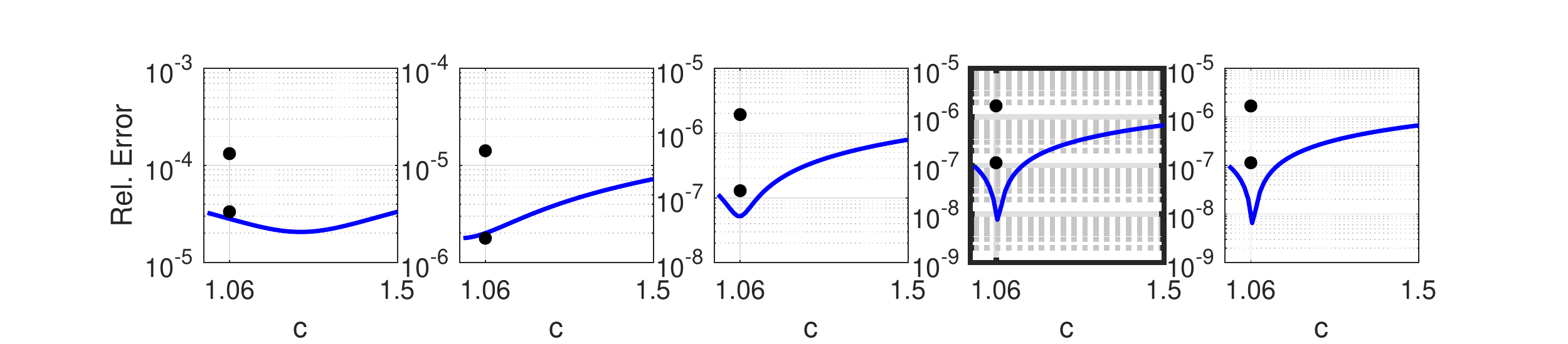}
  \caption{\small{Performance of extrapolation weights $c$ in \eqref{eq:richardson-extrapolation} for different Runge-Kutta schemes and spline interpolation orders: Second, third and forth order RK (top to bottom rows), Spline degrees $p=1,\ldots,5$ (left to right). The value $c^\ast$ is marked on each horizontal axis in the plots. As we proceed left-to-right in each row, we expect that $c= c^\ast$ produces the optimal (smallest) error when the spline order matches the time-stepping order; this happens in the second, third, and fourth columns of rows 1, 2, and 3, respectively. If we increase the spline order beyond the time-stepping order, no change should be observed since the error rate of the multifidelity spline approximations is then limited by the time-stepping error. The black dots represent the errors $||{\hat{w}_2} - u ||_2$,~$||{\hat{w}_3} - u ||_2$ and the blue line is obtained by varying $c$ in $||{\hat{w}^*(c)} - u ||_2$ cf. Equation~\eqref{eq:richardson-extrapolation}. We plot the black dots at the abscissa corresponding to $c = c^\ast$ to better visualize the magnitude of accuracy enhancement. The figures with bold axis lines correspond to those where the time integration order matches the spline order.}}\label{fig_4}
\end{figure}
%
\begin{figure}[h]
\includegraphics[width=0.9\textwidth]{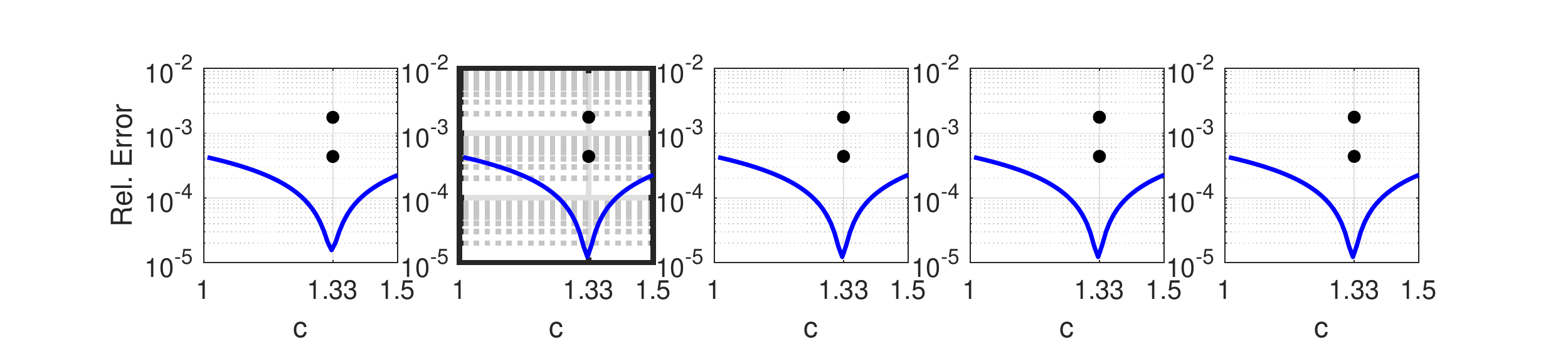}
\includegraphics[width=0.9\textwidth]{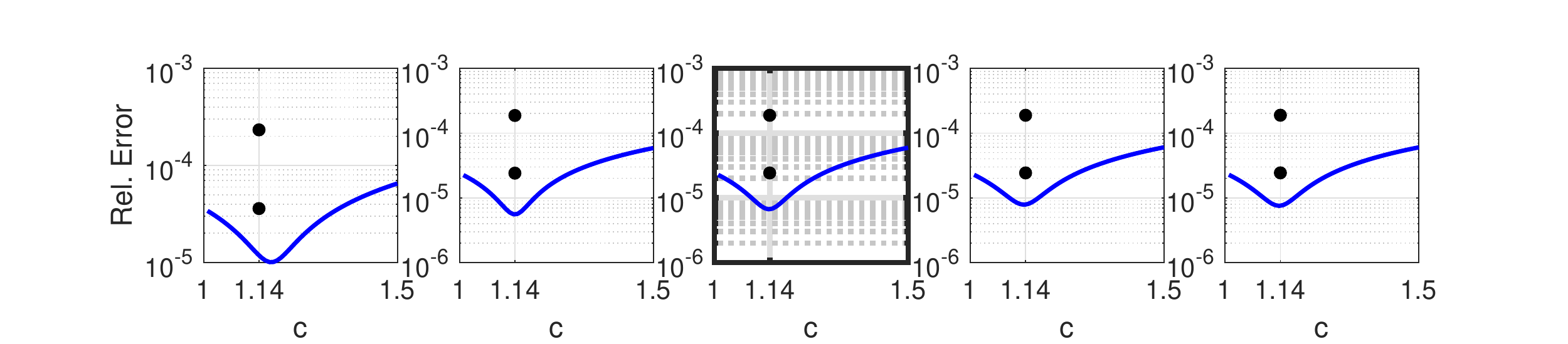}
\includegraphics[width=0.9\textwidth]{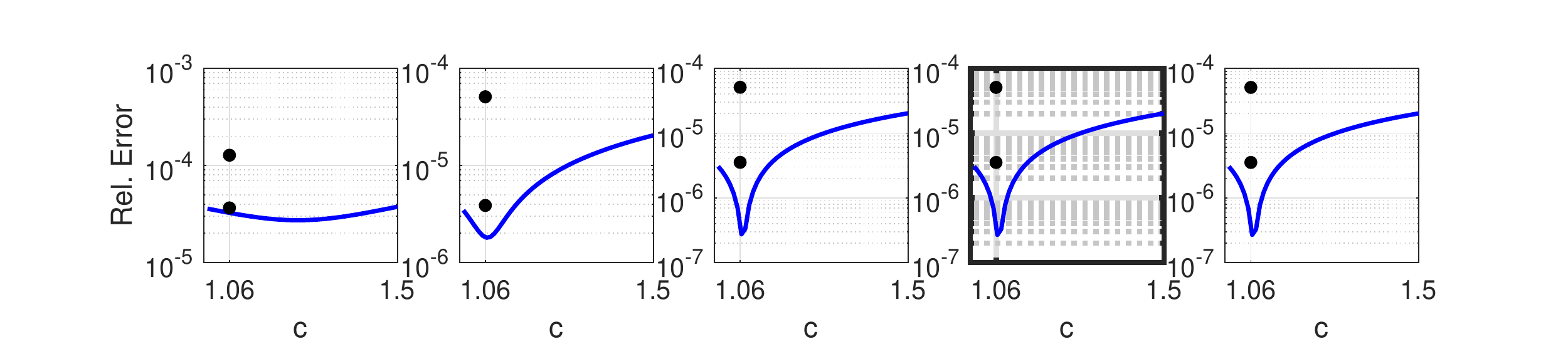}
\caption{\small{Performance of extrapolation weights $c$ in \eqref{eq:richardson-extrapolation} for different Adams-Bashforth schemes and spline interpolation degrees: Second, third and forth order AB (top to bottom rows), Spline degrees $p=1,\ldots,5$ (left to right). See Figure \ref{fig_4}'s caption for a more detailed description.}}\label{fig_5}
\end{figure}

Finally we numerically investigate the order of convergence for the convergence-accelerated multifidelity estimator $\widehat{w}^\ast$ and compare it to the expected order of accuracy, $p$. We observe in Figure \ref{fig:convrate} that convergence rates of order $p+1$ are observed. Owing to Theorem \ref{thm:mf-error}, the multifidelity approximations $\widehat{u}_j$ have error scaling like $h^p( 1 + r^{1-j})$. While a Richardson Extrapolation technique can eliminate the $h^p r^{1-j}$ term, we do not expect that it eliminates the $h^p$ term. We believe this discrepancy in theory and results is due to the explanation in Remark \ref{rem:mf}, i.e., that our estimate in Theorem \ref{thm:mf-error} resulting in a non-$j$-dependent $h^p$ term is a loose bound. As a consequence, we numerically observe order-$(p+1)$ convergence for the accelerated solution instead of the theoretically-expected order-$p$ convergence.

\subsubsection{Statistical moments}\label{sssec:stats}
Under the same model \eqref{linear_ODE} we interpret $k$ as a random variable, uniformly distributed over $[5, 25]$. We can then use the multifidelity procedure to estimate statistical moments of $u(k, t)$ as a function of $t$. In all cases, our statistics are computed numerically using a size-1000 ensemble of Monte Carlo values of $k$. The trajectories of the exact mean and standard deviation of $u$ are shown in the left-hand pane of Figure \ref{fig:meanstd}. In the right-hand pane, we compute the statistical moment error using the multifidelity surrogate $\widehat{w}^\ast$, and compare them against the statistical moment errors computed using the original models $u_1$, $u_2$, and $u_3$. 

There are two ways to compute the moments for the $j=3$ solution: via repeated query of $u_3(\cdot,k)$, or via repeated query of the convergence-accelerated multifidelity approximation $\widehat{w}^\ast(\cdot,k)$. Suppose $X$ is the cost of computing a single solution for $u_3(\cdot, k)$. Then the cost of moments via $u_3$ is $1000 X$. However, the cost of a single online stage query of $\widehat{w}^\ast$ is $X/4$. The offline stage requires $100$ low-fidelity simulations ($100 X/4$), 13 medium-fidelity simulations ($13 X/2$), and 13 high-fidelity simulations ($13 X$). The cumulative cost of $\widehat{w}^\ast$ is thus $295 X$, which is much smaller than the $1000 X$ cost of $u_3$, and $\widehat{w}^\ast$ is also about an order of magnitude more accurate.


\begin{figure}[h]
\includegraphics[width=0.9\textwidth]{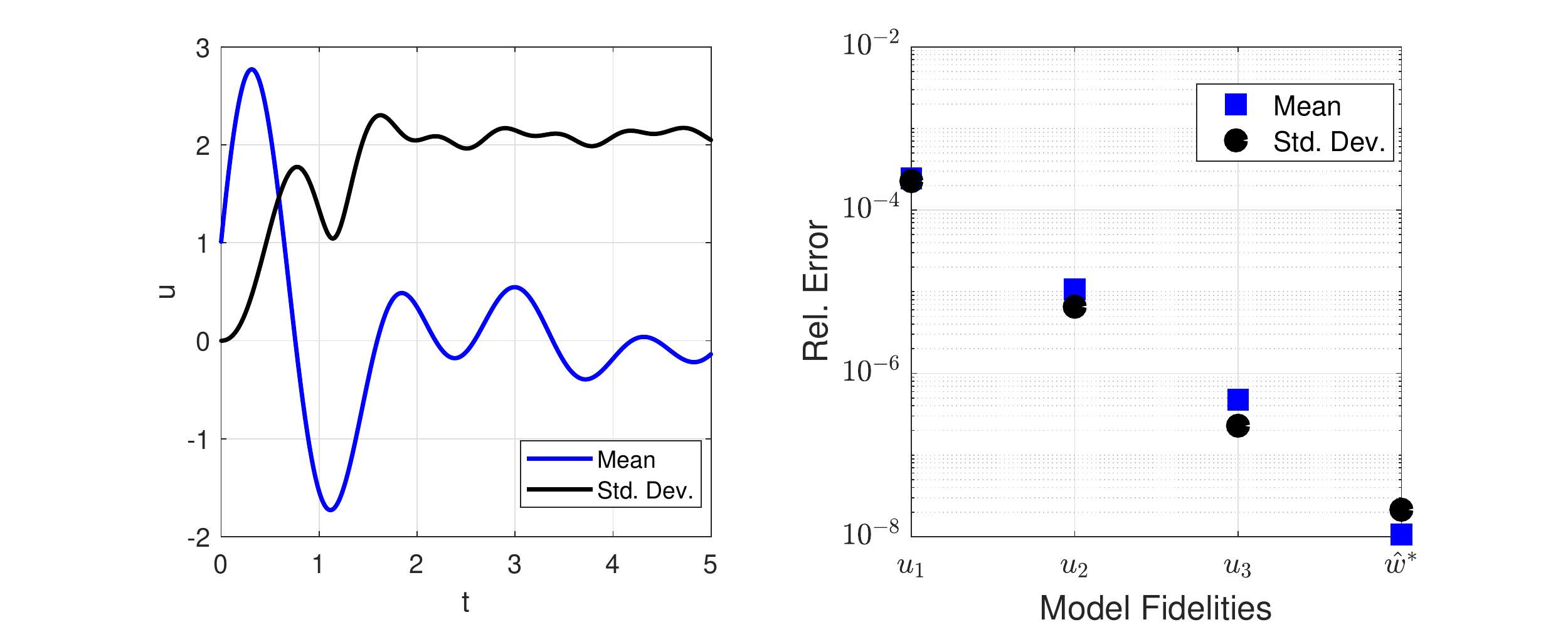}
  \caption{\small{Mean and standard deviation of exact ODE solution (left). Error of convergence-accelerated multifidelity approximation $\widehat{w}^\ast$ vs error from using low-, medium-, or high-fidelity approximations $u_1$, $u_2$, and $u_3$, respectively. Note that the cost of computing the moments of $\widehat{w}^\ast$ is much cheaper than computing the moments of $u_2$ and $u_3$, requiring a total of only 13 medium- and high-fidelity solutions..}}\label{fig:meanstd}
\end{figure}


\begin{figure}[h]
\includegraphics[width=0.9\textwidth]{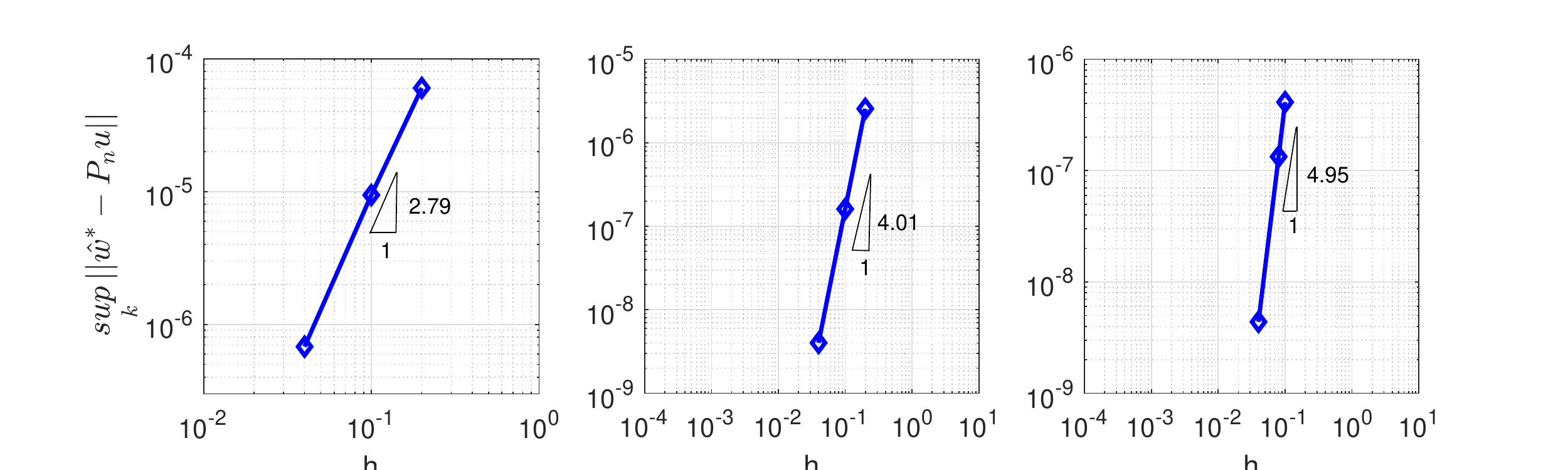}
\caption{\small{Convergence rate corresponding to the slope $p+1$ in the theoretical estimate: $p+1=2.79, 4.01, 4.95$ for RK2 (left), RK3 (middle) and AB4 (right). The $sup$ is taken over $100$ equally spaced points of $k \in [5,25]$.}}\label{fig:convrate}
\end{figure}

\subsection{Nonlinear ODE: Predator-Prey Equations}

We now consider the Lotka-Volterra predator-prey equations. The set of equations is comprised of nonlinear ODEs that are primarily used to describe simplified dynamics of biological systems. The evolution of population for prey and predator species $x(t)$ and $y(t)$, respectively, is modeled as
\begin{align*}
\begin{array}{l}
\vspace{0.1cm}
\displaystyle \frac{\mathd x}{\mathd t}= a x- b xy, \quad x(0)=x_0\\
\displaystyle \frac{\mathd y}{\mathd t}=cxy - d y, \quad y(0)=y_0
\end{array}
\end{align*}
for an initial population $(x_0, y_0) = (1,1)$. We parameterize the positive constants $a, b, c,$ and $d$ by
\begin{align*}
\begin{array}{l l}
\vspace{0.1cm}
a = k + 0.5, & b= 3k+1\\
c = k + 1, & d= k+0.5\\
\end{array}
\end{align*}
where $k$ is a parameter taking values in the interval $[0.5, 1.5]$. In this example since we do not have the analytical solution we solve the nonlinear ODE on a fine grid $h=10^{-3}$ and use that as an ``exact" solution to investigate convergence. 




In Figure \ref{fig_6_1} we show convergence of the unaccelerated multifidelity solutions, $\widehat{u}_j$, as a function of the number of high-fidelity solutions, $n$.  We observe that the error reaches an asymptotic limit as number of higher fidelity solutions are increased; this is expected since for small values of $n$, the time integration error $C h^p$ is greater than the $n$-term projective error $\| u - \mathcal{P}_n u \|$.
\begin{figure}[h]
\centering
\includegraphics[width=0.8\textwidth]{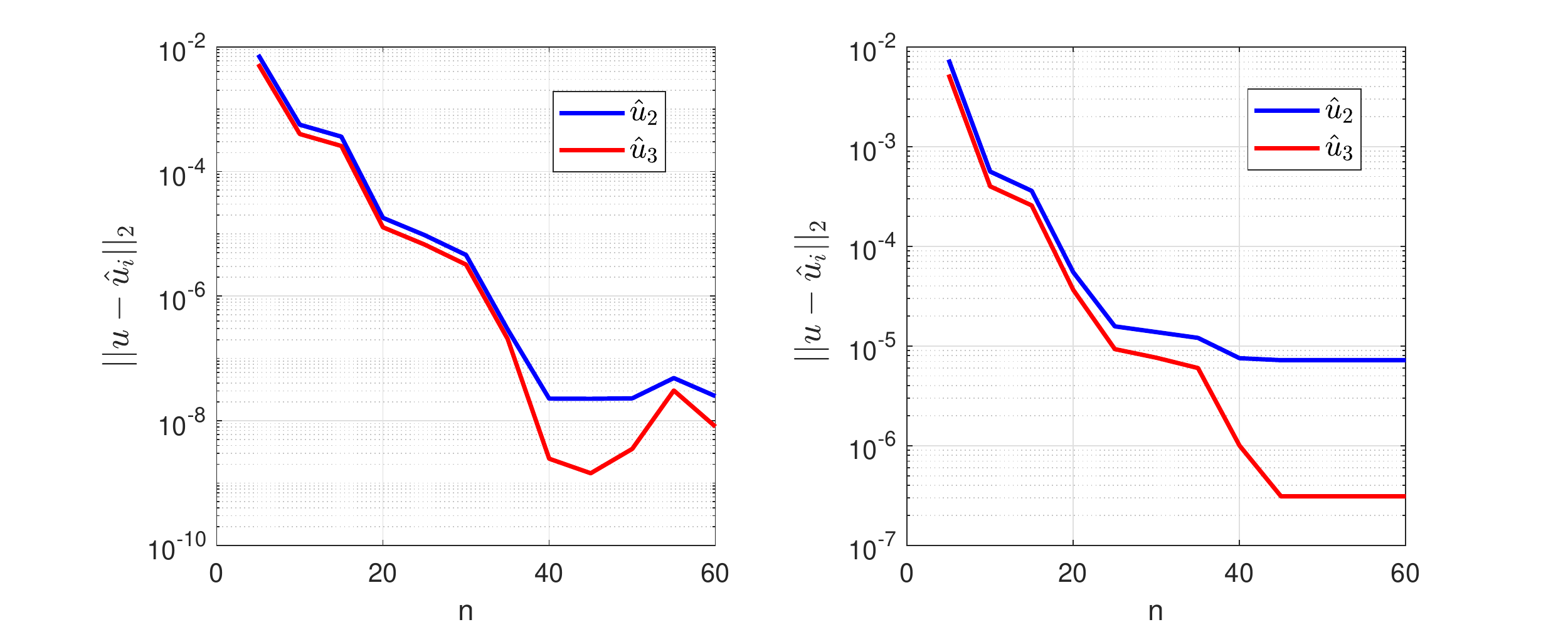}
  \caption{\small{Decay of error between the reconstructed multifidelity and exact solutions with respect to number of higher fidelity solutions: RK4 (left), AB4 (right). }}\label{fig_6_1}
\end{figure}

Figure \ref{fig:pp_convrate} computes numerically-observed rates of convergence for the convergence-accelerated surrogate. We again observe order-$(p+1)$ convergence, despite our theory-based order-$p$ expectation. We again attribute this to the loose bound in our theoretical estimate, as described at the end of Section \ref{sssec:stats}.
\begin{figure}[h]
\includegraphics[width=0.9\textwidth]{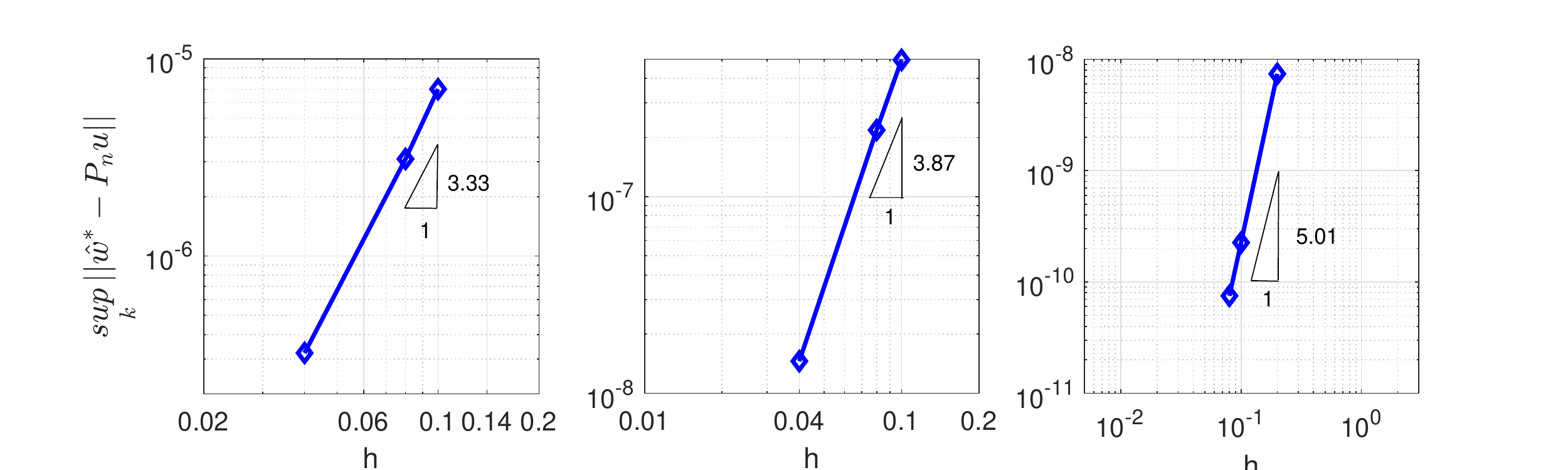}
\caption{\small{Convergence rate corresponding to the slope $p+1$ in the theoretical estimate: $p+1=3.33, 3.87, 5.01$ for AB2 (left), AB3 (middle) and RK4 (right). The $sup$ is taken over $100$ equally spaced points of $k \in [0.5,1.5]$.}}\label{fig:pp_convrate}
\end{figure}

\vspace{-0.3cm}
\section{Concluding Remarks}\label{sec5}

A numerical method for leveraging time-dependent multifidelity models under parametric uncertainty is presented. We built interpolation operators on the inexpensive low-fidelity solution in parameter space, and estimated higher fidelity solutions corresponding at arbitrary parameter locations using the same interpolation rule associated with the low-fidelity solution. We chain this multifidelity procedure together with classical sequence transformation, in particular Richardson extrapolation: Having knowledge of solutions at different fidelity levels allows us to estimate the convergence order
and build a sequence transformation operator that attains superior accuracy compared to the standard multifidelity surrogate.



\end{document}